\documentclass{amsart}

\newtheorem{theorem}{Theorem}[section]
\newtheorem{lemma}[theorem]{Lemma}
\newtheorem{proposition}[theorem]{Proposition}
\theoremstyle{definition}
\newtheorem{definition}[theorem]{Definition}
\newtheorem{example}[theorem]{Example}

\theoremstyle{remark}
\newtheorem{remark}[theorem]{Remark}

\numberwithin{equation}{section}

\begin{document}

\title{On Contact $CR$-Product of Sasakian statistical manifold}

\author[V.Rani,\; J.Kaur]{Vandana Rani$^1$, \;Jasleen Kaur$^2$$^{*}$}
\address{$^{1}$ Department of Mathematics, Punjabi University, Patiala, India.}
\email{vgupta.87@rediffmail.com}

\address{$^{2}$ Department of Mathematics, Punjabi University, Patiala, India.}
\email{jass2381@gmail.com}
\subjclass[2010]{Primary 58A05; Secondary 53D10, 53D15.}

\keywords{Contact $CR$ submanifolds, Sasakian statistical manifold, contact $CR$ product}

\begin{abstract}
 This paper studies the geometric properties of contact $CR$-submanifolds of Sasakian statistical manifold. The integrability of invariant and anti-invariant distributions of contact $CR$-submanifolds has been characterized. Results on $D$-totally geodesic, mixed totally geodesic and $D$-umbilic contact $CR$ submanifolds with regard to dual connections in statistical manifolds have been developed. The statistical version of contact $CR$-product of Sasakian statistical manifold has been introduced. 
\end{abstract}

\maketitle

\section{\textbf{INTRODUCTION}}
The concept of contact  $CR$-submanifolds of a Sasakian manifold  was first introduced by  \cite{yanokoncr} and further developed in \cite{yan01982},  \cite{kon1983}, \cite{yanoandkon1984}. Some fundamental properties of totally umbilical contact $CR$-submanifold of a Sasakian manifold and the contact $CR$-product of a Sasakian space form were worked upon by \cite{matsumoto}. \cite{shahid1985} studied some properties of $D$-totally geodesic and $D^{\perp}$ totally geodesic $CR$-submanifolds of a Sasakian manifold. The examination of these submanifolds was later continued by \cite{shahidtrans} for the Trans-Sasakian manifold which is a generalization of Sasakian manifold. Also skew-$CR$ submanifolds of Sasakian manifold was introduced by \cite{liuximin} where a series of significant results were obtained. Further some related classes such as generalized $CR$ submanifolds of Trans-Sasakian manifolds and $CR$ submanifolds of a Lorentzian para-Sasakian manifold endowed with a semi-symmetric connection were investigated for integrability conditions. \\

Statistical manifolds, introduced by Rao \cite{Rao1945}, have emerged as a prominent area of research in geometry.  The notion of statistical structure has played an important role in the development of a new sphere of study called information geometry. A statistical manifold refers to a statistical model equipped with a Riemannian metric and a pair of dual affine connections. These manifolds have been investigated thoroughly by \cite{amari1985Differential}, \cite{amari1987differential} et.al. and have applications in the field of statistical inference, neural networks, control system etc. In the realm of complex geometry, for a statistical manifold, \cite{furuhata2017} et al. introduced a new notion corresponding to a Sasakian structure, called Sasakian statistical manifold. They obtained some conditions for a real hypersurface in a holomorphic statistical manifold and extended this work for invariant submanifolds in \cite{furuhataii} . This concept was afterwards explored by \cite{kazankazan2018} et al. wherein relations between the curvature tensors of the semi-symmetric metric connection and of the torsion-free dual connections on Sasakian statistical manifold were obtained. \cite{zhang} et.al. studied results on statistical hypersurfaces of Sasakian statistical manifold and holomorphic statistical manifold. \cite{aliyacr2020} also contibuted to the field by developing results on $CR$-statistical submanifolds and establishing characterization theorems for a $CR$-product in holomorphic statistical manifolds.\\

  Since the existing advancements in the geometry of contact $CR$-submanifolds of Sasakian statistical manifold are limited, the present study seeks to examine the structure of these submanifolds. Various results for the integrability and geodesicity of the contact $CR$ submanifolds have been characterized. A necessary and sufficient condition  for a contact  $CR$-submanifold to be a contact $CR$-product has been demonstrated.

\vspace{.2in}
\section{Preliminaries}
Following are some fundamental concepts 
    relevant to the theory of submanifolds of a Sasakian statistical manifold.\\
    \begin{definition}\label{def21}\cite{furuhata2016submanifold}
    Let $\bar M$ be a $C^\infty$ manifold of dimension $\bar m \geq 2$,  $\bar\nabla$ be an affine connection on $\bar M$ and $\bar g$ be a Riemannian metric on $\bar M$.  Then 
     ($\bar M$,$\bar\nabla$,$\bar g$) is called a statistical manifold if
     $\bar\nabla$ is of torsion free and 
    \begin{equation} \label{eq21}
     (\bar\nabla_{X} \bar{g}) (Y,Z) = (\bar\nabla_{Y} \bar{g} )(X,Z) \quad for \;\; X,Y,\in \Gamma(T\bar M).
    \end{equation}
      
    Moreover, an affine connection $\bar\nabla^*$ is called the dual connection of $\bar\nabla$ with respect to $\bar g$ 
    if 
    \begin{equation}\label{eq22}
    X\bar g(Y,Z) = \bar g(\bar\nabla_X Y,Z) + \bar g(Y,\bar\nabla^*_XZ)\quad  for\; X,Y,Z \in \Gamma(T\bar M) \end{equation}
    If ($\bar M$,$\bar\nabla$,$\bar g$) is a statistical manifold, then so is ($\bar M$,$\bar\nabla^*$, $\bar g$). We therefore denote the statistical manifold  by $(\bar{M},\bar{g},\bar{\nabla},\bar{\nabla}^*)$.\\
    \end{definition}
	\begin{remark}	
		Any torsion-free affine connection $\bar{\nabla}$ has a dual connection $\bar{\nabla}^{*}$ given by
		\[
		2\hat{\bar{\nabla}} = \bar{\nabla} + \bar{\nabla}^{*}
		\] 
	 where $\hat{\bar{\nabla}}$ is Levi-Civita Connection. 
		\end{remark}
\begin{remark}
For any vector field $X,Y$ on $\bar{M}$, the difference (1,2) - tensor field $K_XY$ is defined as
\begin{equation}\label{eq23}
K_XY = \bar{\nabla}_XY - \hat{\bar{\nabla}}_XY 
\end{equation}	
 Since $\bar{\nabla}$ and $\widehat{\bar{\nabla}}$ are torsion free, we have
    \[
    	K_{X}Y = K_{Y}X, \qquad \bar{g}(K_{X}Y, Z) = \bar{g}(Y, K_{X}Z).
    \]
    for any $X, Y, Z \in \Gamma(T\bar{M})$.\\
 \noindent Also
    \[
    K(X,Y) = \widehat{\bar{\nabla}}_{X}Y - \bar{\nabla}^{*}_{X}Y
    \]
    From the above equations, we get
    \[
    K(X,Y) = \frac{1}{2}(\bar{\nabla}_{X}Y - \bar{\nabla}^{*}_{X}Y)
    \]
    
\end{remark}
	
\begin{definition}\cite{shahid1985}
For an odd-dimensional Riemannian manifold $(\bar{M}, \bar{g})$, if there exist a $(1,1)$ tensor field $\phi$, a unit vector field $\xi$ called characteristic vector field, and a 1-form $\eta$ such that
    	
    	\begin{equation}\label{eq24}
    		\phi^{2}(X) = - X + \eta(X)\xi, \quad \bar{g}(X, \xi) = \eta(X),
    	\end{equation}
    	and
    	\begin{equation}\label{eq25}
    	    		\bar{g}(\phi X, \phi Y) = \bar{g}(X, Y) - \eta(X)\eta(Y), \qquad \bar{g}(\xi, \xi) = 1,   
    	    	\end{equation} 
then $\bar{M}$ is said to be an almost contact metric manifold with an almost contact metric structure $(\phi, \xi, \eta, \bar{g})$.\\
 If    	    	    	    	
    	\begin{equation}\label{eq26}
    		 d\eta(X,Y) = \bar{g}(X,\phi Y)= -\bar g(\phi X,Y), \quad \forall \;X,Y \in \Gamma(T\bar{M}) ,
    	\end{equation}
    	it follows that $\phi\xi = 0,\; \eta o\phi = 0,\; \eta(\xi) = 1$ and $(\phi, \xi, \eta, \bar{g})$ is called a contact metric structure on $\bar{M}$.\\
    	
    	Also, $\bar{M}$ has a normal contact structure if $N_{\phi} + d\eta \otimes \nu = 0$, where $N_{\phi}$ is the Nijenhuis tensor field.\\
	
 An almost contact metric manifold $\bar{M}$ is called a Sasakian manifold if
	\begin{equation}\label{eq27}
	\hat{\bar{\nabla}}_{X}\xi = - \phi X, 
	\end{equation}
	\begin{equation}\label{eq28}
	(\hat{\bar{\nabla}}_{X}\phi)Y = \bar{g}(X,Y)\xi - \eta(Y)X.
	\end{equation}
	holds for any $X, Y \in \Gamma(TM)$, where $\hat{\bar{\nabla}}$ is Levi-Civita Connection. 
	\end{definition} 

\begin{definition}\cite{furuhata2017}\label{def2.6}
 A quadruplet $(\bar{\nabla} = \hat{\bar{\nabla}} + K, \bar{g}, \phi, \xi)$ is called  Sasakian statistical structure on $\bar{M}$ if\\ 
    	$(i)\;(\bar{g},\phi,\xi)$ is  Sasakian structure on $\bar{M}$\\
    	$(ii)\;(\bar{\nabla},\bar{g})$ is a statistical structure on $\bar{M}$\\
    	and the condition
    	\begin{equation}\label{eq29a}
    		K(X, \phi Y) + \phi K(X, Y) =0
    	\end{equation}
    	holds for any $X, Y \in \Gamma(T\bar{M})$.\\
    	Then $(\bar{M}, \bar{\nabla}, \bar{g},\phi, \xi)$ is called Sasakian statistical manifold. If $(\bar{M}, \bar{\nabla}, \bar{g}, \phi, \xi)$ is  Sasakian statistical manifold, then so is $(\bar{M}, \bar{\nabla}^{*}, \bar{g}, \phi, \xi)$.
\end{definition}
\begin{theorem}\cite{furuhata2017}
	 Let $(\bar{M}, \bar{\nabla}, \bar{g})$ be a statistical manifold and $(\bar{g}, \phi, \xi)$ an almost contact metric structure on $\bar{M}$. Then $(\bar{\nabla}, \bar{g}, \phi, \xi)$ is a Sasakian statistical struture if and only if the following conditions hold:
	
	\begin{equation}\label{eq29}
	\bar{\nabla}_{X}\phi Y - \phi\bar{\nabla}^{*}_{X}Y = \bar g(Y, \xi)X - \bar{g}(Y,X)\xi,
	\end{equation}
	\begin{equation}\label{eq210}
	\bar{\nabla}_{X}\xi = \phi X + \bar g(\bar{\nabla}_{X}\xi, \xi)\xi,
	\end{equation}	
for the vector fields $X, Y$ on $\bar{M}$.	
\end{theorem}
\begin{remark}
For the dual connection $\bar{\nabla}^*$
\[
\bar{\nabla}^*_{X}\phi Y - \phi\bar{\nabla}_{X}Y = \bar g(Y, \xi)X - \bar{g}(Y,X)\xi,
\]
\[
	\bar{\nabla}^*_{X}\xi = \phi X + \bar g(\bar{\nabla}^*_{X}\xi, \xi)\xi
\]
\end{remark}
\begin{remark}
	 By setting
	\[
	K(X, Y) = g(X, \xi )g(Y, \xi )\xi
	\]
	for any $X, Y \in \Gamma(T\bar{M})$ such that $K$ satisfies (\ref{eq23}) and (\ref{eq29a}), we obtain Sasakian statistical structure $(\bar{\nabla}^{\lambda} = \widehat{\bar{\nabla}} + \lambda K, \bar{g}, \phi, \xi)$ on $\bar{M}$ where $\lambda \in C^{\infty}(\bar{M})$,
\end{remark}
Now, let $M$ be a submanifold of Sasakian statistical manifold $\bar{M}$ with induced metric $g$. Then the Gauss and Weingarten formulae are respectively given by
\begin{equation}\label{eq211}
\bar{\nabla}_XY = \nabla_XY + h(X,Y), \;\;\;\; \bar{\nabla}^*_XY = \nabla^*_XY + h^*(X,Y)
\end{equation}
\begin{equation}\label{eq212}
\bar{\nabla}_XV = -A_VX + \nabla^{\perp}_XV, \;\;\;\; \bar{\nabla}^*_XV = -A^*_VX + \nabla^{*\perp}_XV
\end{equation}  
where $\nabla$ and $\nabla^*$ are induced connections on $M$ and dual with respect to $g$. Moreover, $h$ and $h^*$ are symmetric and bilinear, called the embedding curvature tensors of $M$ in $\bar{M}$ for $\bar{\nabla}$ and $\bar{\nabla}^*$ respectively.\\

The second fundamentals forms and shape operators are related by
\begin{equation}\label{eq213}
g(A_VX,Y)= g(h^*(X,Y),V),\;\;\;\; g(A^*_VX,Y)= g(h(X,Y),V)
\end{equation}
for any $X,Y \in\Gamma(TM)$ and $V \in \Gamma(T^{\perp}M)$.\\
For any vector field $X\in\Gamma(TM)$, we have
\begin{equation}\label{eq214}
\phi X = TX + FX
\end{equation} 
where $TX$ and  $FX$ are the tangential and the normal parts respectively of $\phi X$.\\
Similarly, for $V \in \Gamma(T^{\perp}M)$, 
\begin{equation}\label{eq215}
\phi V = BV + CV
\end{equation} 
where $BV$ and $CV$ are respectively the tangential  and the normal parts of $\phi V$.\\
Also, for any $X,Y \in\Gamma(TM)$, we write
\[
g(\phi X,Y) = g(TX,Y)= -g(X,TY)
\]
For any $U,V \in \Gamma(T^{\perp}M)$, we get
\begin{equation}\label{eq216}
g(\phi U,V) = g(CU,V)= -g(U,CV)
\end{equation}
These equations imply that $T$ and $C$ are also skew symmetric tensor fields.\\
Further, for any $X \in\Gamma(TM)$ and $V \in \Gamma(T^{\perp}M)$, we have
\[
g(\phi X,V) = g(FX,V), \;\;\;\; g(X,\phi V) = g(X,BV)
\]
From equation (\ref{eq26}), we have
\begin{equation}\label{eq217}
g(FX,V) = -g(X,BV)
\end{equation}

Now, from (\ref{eq24}), (\ref{eq214}) and (\ref{eq215}), we obtain
\[
T^2X=-X+\eta(X)\xi - BFX, \;\;\;\; FTX = -CFX
\]
and
\[
C^2V = -V -FBV , \;\;\;\; TBV = -BCV
\]
From the equations (\ref{eq29}), (\ref{eq211}), (\ref{eq212}), (\ref{eq214}) and (\ref{eq215}), we have the following results for covariant derivatives of the tensor field $P$, $B$, $F$ and $C$.   
\begin{proposition}\label{prop1}
Let $M$ be a submanifold of Sasakian statistical manifold $\bar{M}$, then we have
\[
\nabla_XTY - T\nabla^*_XY = A_{FY}X + Bh^*(X,Y)+ g(Y,\xi)X - g(Y,X)\xi,
\]
\[
\nabla^{\perp}_XFY - F\nabla^*_XY = Ch^*(X,Y) - h(X,TY)
\]
\[
\nabla_XBV - B\nabla^{*\perp}_XV = A_{CV}X - TA^*_VX
\]
\[
\nabla^{\perp}_XCV - C\nabla^{*\perp}_XV = -h(X,BV) - FA^*_VX
\]
\end{proposition}
\begin{remark}
Since $M$ is tangent to $\xi$, then from equations (\ref{eq210}), (\ref{eq211}) and (\ref{eq214}), we get
\[
\nabla_X\xi = TX + g(\nabla_X\xi, \xi)\xi \quad, \;\;\;\; h(X,\xi) = F(X).
\]
\end{remark}
\section{Contact CR-submanifolds of Sasakian statistical manifold}
Inspired from \cite{Dirik2020}, in this section, we define contact $CR$-submanifolds in a Sasakian statistical manifold and develop some properties on its structure.\\

Let $(M,g)$ be a submanifold of a Sasakian statistical manifold $\bar{M}$ and $(\phi, \xi, \eta)$ be an almost contact structure on $\bar{M}$. Then $M$ is called an invariant submanifold if $\phi(T_xM)\subseteq T_xM$ $\forall \;x \in M$, Further, $M$ is said to be an anti-invariant submanifold if $\phi(T_xM)\subseteq T^{\perp}_xM$ $\forall x \in M$. Similarly, it can be easily seen that a submanifold of an almost contact mertric manifold $\bar{M}$ is said to be invariant (or anti-invariant), if $F$ (or $T$) are identically zero.
\begin{definition}
A submanifold of a Sasakian statistical manifold $\bar M$ is said to be a contact $CR$-submanifold, if there exists on $M$, a differentiable invariant distribution $D$ whose complementary orthogonal distribution $D^{\perp}$ is anti invariant i.e.,\\
(i) $TM = D\oplus D^{\perp}$, $\xi \in \Gamma(D)$\\
(ii) $\phi(D_x) = D_x$\\
(iii) $\phi(D^{\perp}_x) \subseteq T^{\perp}_xM$, for each $x \in M$.\\
A contact $CR$-submanifold is called anti-invariant ( or totally real) if $D_x = 0$ and invariant (or holomorphic) if $ D^{\perp}_x =0$, respectively, for any $x \in M$. It is called proper contact $CR$ - submanifold if neither $D_x = 0$ nor $ D^{\perp}_x = 0$.
\end{definition}
Consider the orthogonal projections $P_1$ and $P_2$ on $D$ and $D^{\perp}$ respectively. Then 
\[
X= P_1X + P_2 X +\eta(X)\xi
\] 
for any $X \in \Gamma(TM)$, where $P_1X \in \Gamma(D)$ and $P_2X \in \Gamma(D^{\perp})$. From (\ref{eq214}), we derive
\[
\phi X = TX + FX \]\[ \phi P_1X + \phi P_2 X= T P_1X + F P_1X + T P_2X + F P_2X
\]
\[
F P_1 = 0,\;\; T P_2 = 0,\;\; F= F P_2,\;\; T= T P_1
\]
We denote the orthogonal subbundle $\phi D^{\perp}$ in $T^{\perp}M$ by $\nu$. Then, we have
\[
T^{\perp}M= \phi D^{\perp} \oplus \nu, \;\; \phi D^{\perp} \perp \nu.
\]
Here, we note that $\nu$ is invariant subbundle with respect to $\phi$.\\

Let $M$ be a contact $CR$-submanifold of a Sasakian statistical manifold. Then we have
\[
-A_{\phi Z}U + \nabla^{\perp}_U \phi Z = \phi(\nabla^*_UZ + h^*(U,Z)) - g(Z,U)\xi
\]
for any vector field $U$ tangent to $M$ and $Z$ in $D^{\perp}$. By virtue of above equation, we get
\[
g(-A_{\phi Z}U,X) = g(\phi\nabla^*_UZ,X) -g(Z,U)g(\xi,X)
\]
for any vector field $U$ tangent to $M$, $X$ in $D$ and $Z$ in $D^{\perp}$. Now, from equation (\ref{eq26}), we obtain
\begin{equation}\label{eq31}
g(A_{\phi Z}U,X) = g(\nabla^*_UZ,\phi X)+\eta(X)g(Z,U)
\end{equation}
\begin{proposition}
For a contact CR-submanifold $M$ of a Sasakian statistical manifold $\bar{M}$ such that $\xi \in \Gamma(D)$, we have
\[
A_{FY} Z = A_{FZ}Y\;\; and \;\; A^*_{FY} Z = A^*_{FZ}Y,\;\; \forall\; Y, Z \in \Gamma(D^{\perp})
\] 
\end{proposition}
\begin{proof} For any $Y, Z \in \Gamma(D^{\perp})$ and $ U \in \Gamma(TM)$ alongwith equations (\ref{eq22}), (\ref{eq26}),(\ref{eq211}) and (\ref{eq213}), we get
\[
\bar g(\bar\nabla_U \phi Z - \phi \bar{\nabla}^*_UZ,Y )= \bar g(A_{FY}Z,U) - \bar g(A_{FZ}Y,U)
\]
By using Sasakian statistical maniofold the above equation implies 
\[
\bar g(A_{FY}Z,U) - \bar g(A_{FZ}Y,U)= 0
\]
Hence the result. Similarly, the corresponding case for the dual connection also holds.
\end{proof}
\begin{proposition}
Let $M$ be a contact CR-submanifold of a Sasakian statistical manifold. Then we have
\[
A^*_{U} BV = A^*_{V}BU
\]
 $\forall ~~ U,V \in \Gamma(T^{\perp}M)$ if and only if
 \[
 \nabla^{\perp}_X CV = C \nabla^{*\perp}_X V
\] 
for any $X \in \Gamma(TM)$ and $V \in \Gamma(T^{\perp}M)$
\end{proposition}
\begin{proof} From proposition (\ref{prop1}), (\ref{eq213}) and (\ref{eq217}),   we have
 \begin{eqnarray}
g(\nabla^{\perp}_X CV, U) -g(C\nabla^{*\perp}_XV,U)&=& -g(A^*_U BV,X) + g(A^*_VX,BU)\nonumber \\
&=& -g(A^*_U BV,X) + g(A^*_V BU,X) \nonumber
\end{eqnarray}
Using the self adjoint property of $A^*$, our assertion follows.
\end{proof}
\begin{proposition}
For a contact CR-submanifold $M$ of a Sasakian statistical manifold $\bar{M}$, we have
\[
\nabla^{\perp}_X FY = F \nabla^{*}_X Y
\]
 for any $X,Y \in \Gamma(TM)$ iff
 \[
 \nabla_X BV = B \nabla^{*\perp}_X V
\] 
for any $X \in \Gamma(TM)$ and $V \in \Gamma(TM^{\perp})$.
\end{proposition}
\begin{proof} Using proposition (\ref{prop1}) and  equations (\ref{eq213}) and (\ref{eq216}),  we have
\begin{eqnarray}
g(\nabla^{\perp}_X FY, V) -g(F\nabla^*_XY,V) &=& -g(h^*(X,Y),CV) + g(A^*_VX,TY) \nonumber\\
&=& -g(A_{CV} X,Y) + g(TA^*_VX,Y) \nonumber \\&=& -g(\nabla_X BV,Y) + g(B \nabla^{*\perp}_X V,Y)\nonumber
\end{eqnarray}
from the skew symmetric property  of $T$ and proposition (\ref{prop1}). Thus our assertion follows.

\end{proof}

\begin{theorem}\label{th35}
Let $M$ be a contact CR-submanifold of a Sasakian statistical manifold $\bar{M}$. Then the invariant distribution $D$ is integrable  if and only if the second fundamental form of $M$ satisfies:
\[
\bar{g}(h(X,\phi Y),\phi Z) = \bar{g}( h(Y,\phi X),\phi Z)
\]
for any $X, Y \in \Gamma(D)$ and $Z \in \Gamma (D^{\perp})$.
\end{theorem}
\begin{proof} For any vector field $X, Y \in \Gamma(D)$, using equations (\ref{eq211}) in (\ref{eq29}), we obtain
\[
\nabla_X \phi Y + h(X,\phi Y) - \phi\nabla^*_XY - \phi h^*(X,Y) = g(Y,\xi)X - g(Y,X)\xi
\]
Now, interchanging $X$ and $Y$ and then by subtracting, we get
\[
\nabla_X \phi Y - \nabla_Y \phi X +  h(X,\phi Y) -  h(Y,\phi X) - \phi\nabla^*_XY + \phi\nabla^*_YX = g(Y,\xi)X - g(X,\xi)Y
\]
\[
\phi [X,Y] = \nabla_X \phi Y - \nabla_Y \phi X +  h(X,\phi Y) -  h(Y,\phi X)- g(Y,\xi)X + g(X,\xi)Y
\]
Using (\ref{eq214}), and on comparing the normal components, we have
\[
F[X,Y] = h(X,\phi Y) -  h(Y,\phi X)
\]
Thus $D$ is integrable if and only if $h(X,\phi Y) =  h(Y,\phi X)$.
Also, 
\[
\bar{g}(h(X,\phi Y),\phi Z) = \bar{g}( h(Y,\phi X),\phi Z)
\]
for any $X, Y \in \Gamma(D)$ and $Z \in \Gamma (D^{\perp})$.
\end{proof}
\begin{theorem}
Let $M$ be a contact CR-submanifold of a Sasakian statistical manifold $\bar{M}$. Then the anti- invariant distribution $D^{\perp}$ is integrable  if and only if 
\[
A_{\phi Y}X - A_{\phi X}Y = g(Y,\xi)X - g(X,\xi)Y
\]
for any $X, Y \in \Gamma(D^{\perp})$.
\end{theorem}
\begin{proof} For any vector field $X, Y \in \Gamma(D^{\perp})$ and from (\ref{eq212}) 
\[
\bar \nabla _X \phi Y = -A_{\phi Y}X + \nabla^{\perp} _X \phi Y
\] 
The equations (\ref{eq212}), (\ref{eq214}), (\ref{eq215}) and (\ref{eq29}) yield that
\[
\nabla^{\perp}_X \phi Y = -A_{\phi Y}X + T\nabla^*_X Y + F\nabla^*_X Y + Bh^*(X,Y) + Ch^*(X,Y) + g(Y, \xi)X - g(Y,X)\xi
\]
\end{proof}
By equating tangential components on both sides, we obtain
\[
A_{\phi Y}X  = T\nabla^*_X Y + Bh^*(X,Y)+ g(Y, \xi)X - g(Y,X)\xi
\]
On interchanging $X$ and $Y$ in above equation and then subtracting, we conclude that
\[
A_{\phi Y}X - A_{\phi X}Y = T[X,Y] + g(Y,\xi)X - g(X,\xi)Y
\]
Thus the proof follows from the hypothesis.\\

\begin{definition}
	A contact $CR$-submanifold $M$ of a Sasakian statistical manifold $\bar{M}$ is called ${D}$ - totally geodesic contact $CR$- submanifold with respect to $\bar{\nabla}$ (respectively $\bar{\nabla}^{*}$) if its second fundamental form  satisfies $h(X,Y)=0$ (respectively $h^{*}(X,Y)=0$), for any $X,Y$ $\in$ $\Gamma({D})$.
\end{definition}
\begin{definition}
		A contact $CR$-submanifold $M$ of a Sasakian statistical manifold $\bar{M}$ is called ${D^{\perp}}$ - totally geodesic contact $CR$- submanifold with respect to $\bar{\nabla}$ (respectively $\bar{\nabla}^{*}$) if its second fundamental form  satisfies $h(X,Y)=0$ (respectively $h^{*}(X,Y)=0$), for any $X,Y$ $\in$ $\Gamma({D^{\perp}})$.
\end{definition}
\begin{theorem}\label{theorem1}
Let $M$ be a contact CR-submanifold of a Sasakian statistical manifold $\bar{M}$. Then $M$ is $D$-totally geodesic w.r.t $\bar{\nabla}$ (resp. $\bar{\nabla}^*$) if and only if $A_VX \in \Gamma(D^{\perp})$ (resp. $A^*_VX \in \Gamma(D^{\perp}))$ for each $X \in D$ and $V$ a normal vector field to $M$.
\end{theorem}
\begin{proof}
Let $M$ be totally geodesic  w.r.t $\bar{\nabla}$. From (\ref{eq213}), it follows that
\[
 g(A_V X,Y) = 0
\]
This implies that $A_VX \in \Gamma(D^{\perp})$ for each $X \in D$, $V$ is a normal vector field to $M$.\\

Conversely, suppose that $A_VX \in \Gamma(D^{\perp})$. Then by the similar assertion,  for $X,Y \in D$,\;\;$g(h(X,Y),V) =0$
. Hence $M$ is totally geodesic  w.r.t $\bar{\nabla}$. Thus the result holds for  $\bar{\nabla}^{*}$ also.
\end{proof}
Similarly, we have the following result:
\begin{theorem}
Let $M$ be a contact CR-submanifold of a Sasakian statistical manifold $\bar{M}$. Then $M$ is $D^{\perp}$-totally geodesic w.r.t $\bar{\nabla}$ (resp. $\bar{\nabla}^*$) if and only if $A_VX \in \Gamma(D)$ (resp. $A^*_VX \in \Gamma(D)$ )for each $X \in D^{\perp}$, $V$ a normal vector field to $M$.
\end{theorem}

\begin{definition}
A contact $CR$-submanifold $M$ of a Sasakian statistical manifold $\bar{M}$ is called ${D}$-umbilic contact $CR$- submanifold with respect to $\bar{\nabla}$ (respectively $\bar{\nabla}^{*}$) if  $h(X,Y)=g(X,Y)L$ (respectively $h^{*}(X,Y)= g(X,Y)L$), for any $X,Y$ $\in$ $\Gamma({D})$, $L$ being some normal vector field.
\end{definition}
\begin{theorem}
Let $M$ be a contact CR-submanifold of a Sasakian statistical manifold $\bar{M}$. Then $M$ is $D$-umbilic w.r.t $\bar{\nabla}$ (resp. $\bar{\nabla}^*$), then  $M$ is $D$-totally geodesic w.r.t $\bar{\nabla}$ (resp. $\bar{\nabla}^*$)
\end{theorem}
\begin{proof}
Let $M$ be  ${D}$-umbilic contact CR- submanifold with respect to $\bar{\nabla}$, then 
\[
h(X,Y)=g(X,Y)L \;\;\;\;\forall\;\; X,Y \in \Gamma (D)
\]
 ;$L$ being some normal vector field on $M$.\\
Now on putting $X =Y =\xi $ and using $h(\xi,\xi) = 0$, we have $L= 0$. \\
This  implies that,  $M$ is $D$-totally geodesic w.r.t $\bar{\nabla}$. The corresponding result for dual connection also holds.

\end{proof}

\begin{definition}
	A contact $CR$-submanifold $M$ of a Sasakian statistical manifold $\bar{M}$ is called mixed totally geodesic contact $CR$- submanifold with respect to $\bar{\nabla}$ (respectively $\bar{\nabla}^{*}$) if its second fundamental form  satisfies $h(X,Y)=0$ (respectively $h^{*}(X,Y)=0$), for any $X$ $\in$ $\Gamma({D})$ and $Y$ $\in$ $\Gamma({D^{\perp}})$.
\end{definition}
\begin{theorem}
Let $M$ be a contact CR-submanifold of a Sasakian statistical manifold $\bar{M}$. Then $M$ mixed totally geodesic w.r.t $\bar{\nabla}$ (resp. $\bar{\nabla}^*$) if and only if 
\[
(i)\;  A_VX \in D \;\;\;\;(resp.\;\; A^*_VX \in D ) \;\;\;\; \forall\; X \in D,
\]
\[
(ii) \; A_VX \in D^{\perp} \;\;\;\;(resp.\;\; A^*_VX \in D^{\perp}) \;\;\;\; \forall\; X \in D^{\perp}.
\]
and normal vector field $V$.
\end{theorem}
\begin{proof}
Let $M$ be mixed geodesic  w.r.t $\bar{\nabla}$, then by equation (\ref{eq213}), we have
\[
h(X,Y) = 0
\]
iff $A_VX \in D$, $\forall \;X \in D$ and the normal vector field $V$.\\
Also, part (ii) holds for $X \in D^{\perp}$ and the normal vector $V$ from (\ref{eq213}).
\end{proof}
\begin{theorem}
Let $M$ be mixed totally geodesic  contact CR-submanifold w.r.t $\bar{\nabla}$ (resp. $\bar{\nabla}^*$) of a Sasakian statistical manifold $\bar{M}$. Then,
\[
(i) \;A_{\phi V}X = \phi A^*_V X, \;\;\;\; (resp. \;\; A^*_{\phi V}X = \phi A_V X )
\]  
\[
(ii) \;\nabla^{\perp}_X \phi V = \phi \nabla^{*\perp}_X V, \;\;\;\; (resp.\;\; \nabla^{*\perp}_X \phi V = \phi \nabla^{\perp}_X V )
\]
$\forall X \in D$ and the normal vector field $V$.
\end{theorem}
\begin{proof} For any $X \in D$ and normal vector field $V$, equation (\ref{eq212}) implies
\[
\bar{\nabla}_X \phi V = -A_{\phi V}X + \nabla^{\perp}_X \phi V
\]
Now, by using (\ref{eq29}) and (\ref{eq212}) , we obtain
\[
\bar{\nabla}_X \phi V = \phi(-A^*_VX +\nabla^{\perp}_X V) + g(Y,\xi)X - g(Y,X)\xi
\]
For $X \in D$ and $Y \in D^{\perp}$, the above equations yield that
\[
-A_{\phi V}X + \nabla^{\perp}_X \phi V = \phi(-A^*_VX +\nabla^{\perp}_X V)
\]
By equating tangential and normal parts, we get the desired result.
\end{proof}
\begin{definition}
A contact $CR$-submanifold $M$ of a Sasakian statistical manifold $\bar{M}$ is said to be foliate  contact $CR$- submanifold of $\bar{M}$ if $D$ is involutive.
\end{definition}
\begin{remark}
If $M$ is foliate  contact $CR$-  submanifold, then from equation (\ref{eq24}) we have
\[
h(\phi X, \phi Y) = h(\phi^2X,Y)= -h(X,Y)
\]
\end{remark}
\begin{theorem}
If $M$ is foliate mixed totally geodesic contact CR-submanifold  w.r.t $\bar{\nabla}$ (resp. $\bar{\nabla}^*$) of a Sasakian statistical manifold  $\bar{M}$, then we have 
\[
A^*_V \phi X + \phi A^*_V X =0, \;\;\;\; (resp.\;\; A_V \phi X + \phi A_V X =0 )
\]
$\forall ~~ X \in D $ and normal vector field $V$.
 \end{theorem}
\begin{proof} For $X \in D $ and normal vector field $V$, we have
\begin{eqnarray}
g(A^*_V \phi X,Y) &=& g(h(\phi X,Y), V) = g(h(X,\phi Y), V) \nonumber \\
&=& g(A^*_VX,\phi Y) = -g(\phi A^*_VX,Y) \nonumber
\end{eqnarray}
 using the  given hypothesis and equations (\ref{eq213}) and (\ref{eq26}).
Thus we get the  result w.r.t $\bar{\nabla}$.
\end{proof}
Following \cite{Dirik2020} and \cite{kazankazan2018}, an example of a proper contact $CR$-submanifold of Sasakian statistical manifold is presented as below:
\begin{example}
Let $M$ be a submanifold of $\bar{M}=(R^7,\bar{g})$ defined by the following equation 
\[
\gamma(u,v,w,s,z) = (u,v,w+s,0,0,s-w,z)
\]
The tangent bundle of $M$ is spanned by the tangent vectors 
\[
e_1 =\frac{\partial}{\partial x_1},\;\;\; e_2 =\frac{\partial}{\partial y_1} ,\;\;\; e_3 =\frac{\partial}{\partial x_2} - \frac{\partial}{\partial y_3}, \;\;\; e_4 =\frac{\partial}{\partial x_2} + \frac{\partial}{\partial y_3}, \;\; e_5 =\frac{\partial}{\partial z} 
\]
which are linear independent at each point of $M$.
Let $\bar{g}$ be the Riemannian metric defined by $\bar{g}(e_i,e_j) = 0 $, $i \neq j$, $i = j= 1,2,3,4,5$ and $\bar{g}(e_k,e_k) = 1, k = 1,2,3,4,5.$
Let $\eta$ be the 1-form defined by $ \eta(X) = \bar{g}(X,e_5)$ for any $X \in \Gamma(T\bar{M})$.\\

Let $\phi$ be a $(1,1)$ tensor field  in $R^7$ with coordinate system 
$(x_1,y_1, x_2,y_2,\\ x_3, y_3,z)$, therefore choosing 
\[ \phi(\frac{\partial}{\partial x_i}) = \frac{\partial}{\partial y_i}, \;\;\;  \phi(\frac{\partial}{\partial y_j}) = - \frac{\partial}{\partial x_j},\;\;\;\; (i,j= 1,2,3)
\] we have 
\[
\phi e_1 = e_2, \;\; \phi e_2 = - e_1,\;\; \phi e_3 = \frac{\partial}{\partial y_2} + \frac{\partial}{\partial x_3}, \;\; \phi e_4 = \frac{\partial}{\partial y_2} -\frac{\partial}{\partial x_3}, \;\; \phi e_5 = 0.
\] 

 The linearity of $\phi$ and $\bar{g}$ implies that $\eta(e_5) = 1$, $\phi^2 X = -X + \eta(X)e_5$ and $\bar{g}(\phi X, \phi Y) = \bar{g}(X,Y) -\eta(X)\eta(Y) $ for any $X,Y \in \Gamma(TM)$. Therefore for $e_5 = \xi,$ $(\phi,\xi, \eta,\bar{g})$ defines an almost contact metric structure on $M$.
We have 
\[
[e_1,e_2] = [e_1,e_3] = [e_1,e_4] = [e_1,e_5] =0,\;\; [e_2,e_3] =[e_2,e_4] = [e_2,e_5] =0, 
\]
\[
 [e_3,e_4] = [e_3,e_5] =0, \;\; [e_4,e_5] =0
\]
The Levi-Civita connection $\hat{{\bar \nabla}}$ of the metric tensor $\bar{g}$ is given by the Koszul's formula as 
\[ 
\hat{{\bar \nabla}}_{e_i} {e_j} = 0 \;\; \forall \;i,j = 1,2,3,4,5.
\]
It concludes that $(\phi,\xi, \eta,\bar{g})$ is a Sasakian structure on $\bar{M}$ using (\ref{eq28}). So, $(\bar{M},\phi,\xi, \eta,\bar{g})$ is a  7-dimensional Sasakian manifold.\\

To add statistical structure to $\bar{M}$, suppose a connection $\bar{\nabla}$ is defined as $\bar{\nabla}_XY = \hat{{\bar \nabla}}_XY + K(X,Y) $ where $K(X,Y) = \eta(X)\eta(Y)\xi$. Then
 the torsion tensor of the connection $\bar{\nabla}$, i.e. $T(e_i,e_j) = 0$ and $ (\bar{\nabla}_{e_i}\bar{g})(e_j,e_k) = 0$  for all $i,j,k$. Therefore, $(\bar{\nabla}, \bar{g})$ is a statistical structure. Since $K(e_i,\phi e_j) = \phi K(e_i,e_j) = 0$ for $i,j$, ($\bar \nabla,\bar{g},\phi,\eta,\xi$) is a Sasakian statistical structure on $\bar{M}$. \\

Now $D$ = span$\{e_1,e_2,e_5\}$ is a invariant distribution. Since $\bar{g}(\phi e_3,e_i) = 0$; $i= 1, 2,4,5$ and $\bar{g}(\phi e_4,e_j) = 0$; $j= 1, 2,3,5$, therefore  $\phi e_3$, $\phi e_4 $ are orthogonal to $M$ and $D^{\perp}$ = span$\{e_3,e_4\}$ is an anti-invariant distribution. Thus $M$ is a 5-dimensional proper contact $CR$-submanifold of Sasakian statistical manifold $\bar{M}$.  
\end{example}
\section{Contact $CR$-Product of Sasakian Statistical Manifold}
In this section, the statistical version of $CR$-product in Sasakian statistical manifold has been investigated. A necessary and sufficient condition for a contact  $CR$-submanifold to be a contact $CR$-product has been given. \\

\begin{definition}
A contact $CR$-submanifold $M$  of Sasakian statistical manifold is said to be a contact $CR$-product if it is the locally product of $M^{\top}$ and $ M^{\perp}$, where $M^{\top}$ denotes the leaf of the distribution $D$.
\end{definition}
\begin{remark}
In a contact $CR$-submanifold of Sasakian statistical manifold, we assume that:\\
(i) the leaf of $ M^{\perp}$ of $D^{\perp}$ is totally geodesic w.r.t $\bar{\nabla}$ (resp.$\bar{\nabla}^*$)  in $M$, that is, $\nabla_ZW$ (resp.$\nabla^*_ZW$) is in  $D^{\perp}$  for any $Z,W$ in   $D^{\perp}$.\\
(ii) the leaf of $M^{\top}$ of $D$ is totally geodesic w.r.t $\bar{\nabla}$ (resp.$\bar{\nabla}^*$)  in $M$, that is, $\nabla_ZW$(resp. $\nabla^*_ZW$) is in $D$ for any $Z,W$ in   $D$ . 
\end{remark}
\begin{lemma}\label{lemma43}
Let $M$ be a contact CR-submanifold of Sasakian statistical manifold $\bar{M}$.Then
\[
g(h^*(X,U),\phi Z) = \eta(X)g(\phi Z, \phi U)
\]
if and only if the leaf $M^{\perp}$ of $D^{\perp}$ is totally geodesic in $M$
for any $Z, W \in \Gamma(D^{\perp})$ and $X$ in $D$.
\end{lemma}
\begin{proof} For any $X$ in $D$ and $Z, W \in \Gamma(D^{\perp})$ 
\[
g(A_{\phi Z}U,X) = g(\nabla^*_UZ,\phi X)+\eta(X)g(Z,U)
\]
using equation (\ref{eq31}).
Now $Z, U \in \Gamma(D^{\perp})$ implies that $\nabla^*_UZ$ $\in \Gamma(D^{\perp})$. So from above equation, we conclude that
\begin{equation}\label{eq41b}
g(A_{\phi Z}U,X) = \eta(X)g(Z,U)
\end{equation}
Since $M$ is a contact $CR$ submanifold of $\bar{M}$, the desired result follows from the self adjoint property of $A$ and (\ref{eq213}).
\end{proof}
\begin{lemma}
Let $M$ be a contact CR-submanifold of Sasakian statistical manifold $\bar{M}$.Then, we have
\[
\nabla^{\perp}_Z\phi W- \nabla^{\perp}_W\phi Z \in \phi D^{\perp}
\]
for any $Z, W \in \Gamma(D^{\perp})$ and $\lambda$ in $\nu $.
\end{lemma}
\begin{proof} Using equations (\ref{eq29}),(\ref{eq211}) and (\ref{eq212}), we get
\[
-A_{\phi W}Z + \nabla^{\perp}_Z\phi W = \phi \nabla^*_ZW + \phi h^*(Z,W)+g(W,\xi)Z -g(W,Z)\xi
\]
for $Z, W \in \Gamma(D^{\perp})$ and $\lambda$ in $\nu $.\\

Interchanging W by Z, we get
\[
-A_{\phi Z}W + \nabla^{\perp}_WJZ = \phi \nabla^*_WZ + \phi h^*(W,Z)+g(Z,\xi)W -g(Z,W)\xi
\]
On subtracting and taking inner product with $\lambda$, we get the desired result.
\end{proof}
\begin{lemma}
Let $M$ be a contact CR-submanifold of Sasakian statistical manifold $\bar{M}$.Then, we have
\[
A^*_{\lambda} \phi Y = - A_{\phi \lambda}Y
\]
for any $X,Y\in \Gamma(D)$ and $\lambda$ in $\nu $.
\end{lemma}
\begin{proof}For any $X,Y\in \Gamma(D)$ and $\lambda$ in $\nu $, we have
\[
\nabla_X\phi Y + h(X,\phi Y) - \phi \nabla^*_XY -\phi h^*(X,Y) = g(Y,\xi)X - g(X,Y)\xi
\]
from equations (\ref{eq29}) and (\ref{eq211}).\\

By taking inner product with $\lambda$
\[
\bar{g}(A^*_{\lambda} \phi Y,X) = \bar{g}(- A_{\phi \lambda}Y,X)
\]
follows from (\ref{eq26}) and (\ref{eq213}).
\end{proof}
\begin{theorem}
A contact CR-submanifold $M$  of Sasakian statistical manifold $\bar{M}$ is called a contact CR-product if and only if
\begin{equation}\label{eq41a}
A_{\phi U}X = \eta(X)U
\end{equation}
for any $X$ in $D$ and $Z$ in $D^{\perp}$.
\begin{proof} For any $X$ in $D$ and $Z$ in $D^{\perp}$, (\ref{eq41a}) implies
\[
g(A_{\phi U}X,Z) = g(\eta(X)U,Z)
\]
From equation (\ref{eq213}), we have
\[
g(h^*(X,Z),\phi U) = g(\eta(X)\phi Z,\phi U)
\]
Therefore, Lemma (\ref{lemma43}) implies that the leaf $M^{\perp}$ of $D^{\perp}$ is totally geodesic in $M$
for any $Z, U \in \Gamma(D^{\perp})$ and $X$ in $D$. Also we have
\[
g(h^*(X,\phi Y),\phi Z) = g(\eta(X)Z,\phi Y) = 0
\]
for any $X, Y$ in $D$ and $ Z$ in $D^{\perp}$. So by Theorem (\ref{th35}) , the distribution $D$ is integrable.\\
Let $M^{\top}$ be the leaf of the distribution $D$, then we have
\begin{eqnarray}
g(\nabla^*_XY,Z) =g(\bar{\nabla}^*_XY,Z) &=& g(\phi \bar{\nabla}^*_XY, \phi Z )\nonumber \\
&=& g(\bar{\nabla}_X \phi Y,\phi Z)- g(X,Y)g(\xi,\phi Z) - g(Y, \xi)g(X,\phi Z) \nonumber
\end{eqnarray}
for any $X$, $Y$ in $D$ and $Z$ in $D^{\perp}$. This implies that 
\[
g(\nabla^*_XY,Z) = g(h(X,\phi Y), \phi Z) = g(A_{\phi Z }X,\phi Y) = 0
\]
This means,the leaf $M^{\top}$ is totally geodesic in distribution $D$. Thus the submanifold $M$ of Sasakian statistical manifold is a contact CR product.\\

Conversely, if the submanifold $M$ is contact CR-product, then from (\ref{eq41b}), we have $A_{\phi U}X - \eta(X)U \in  D$ for any $X$ in $D$ and $U$ in $D^{\perp}$. Now, it is sufficient to prove $A_{\phi U}X - \eta(X)U \in  D^{\perp}$ for any $X$ in $D$ and $U$ in $D^{\perp}$. 
Since the distribution $D$ is totally geodesic in $M$, we have 
\[
g(A_{\phi U}X - \eta(X)U, Y) = g(h^*(X,Y), \phi U)  = -g(\phi h^*(X,Y), U) 
\]
Since $M$ is a Sasakian statistical manifold, \[
g(A_{\phi U}X - \eta(X)U, Y) = -g(\phi \bar \nabla^*_XY, U) =- g(\nabla_X \phi U, Z) = 0
\]using Gauss formula.
So we get the desired result. 
\end{proof}
\end{theorem}

\end{document}